\documentclass{amsart}
\usepackage{amsmath}
\usepackage{amsthm}
\usepackage{amssymb}
\usepackage{amsfonts}
\usepackage{graphicx}
\usepackage{enumitem}
\usepackage{color}
\usepackage{verbatim}
\usepackage{multirow}

\newcommand{\e}{\varepsilon}

\newcommand{\R}{\ensuremath{\mathbb{R}}}

\newtheorem {theorem} {Theorem} 
\newtheorem {proposition} [theorem] {Proposition}

{\theoremstyle{definition}

\newtheorem {remark} {Remark}
}

\begin{document}

\author[F. Braun and A.C. Mereu]{Francisco Braun and Ana C. Mereu}

\address{Departamento de Matem\'atica, Universidade Federal de S\~ao Carlos, 13565-905 S\~ao Carlos, S\~ao Paulo, Brazil} 
\email{franciscobraun@dm.ufscar.br}

\address{Departamento de F\'isica, Qu\'imica e Matem\'atica, Universidade Federal de S\~ao Carlos, 18052-780 Sorocaba, S\~ao Paulo, Brazil} 
\email{anamereu@ufscar.br}

\title[Zero-Hopf bifurcation in a 3-D jerk  system]
{Zero-Hopf bifurcation in a 3-D jerk  system}


\maketitle

\begin{abstract}
We consider the 3-D system defined by the jerk equation $\dddot{x} = -a \ddot{x} + x \dot{x}^2 -x^3 -b x + c \dot{x}$, with $a, b, c\in \R$. 
When $a=b=0$ and $c < 0$ the equilibrium point localized at the origin is a zero--Hopf equilibrium. 
We analyse the zero--Hopf Bifurcation that occur at this point when we persuade a quadratic perturbation of the coefficients, and prove that one, two or three periodic orbits can born when the parameter of the perturbation goes to $0$. 
\end{abstract}

\section{Introduction}

Motivated by the development of the Chua circuit \cite{C}, many researchers have been interested in finding other circuits that chaotically oscillate. Some simple third-order ordinary differential equations of the form 
$$
\dddot{x} = J(\ddot{x},\dot{x},x,t), 
$$ 
whose solutions are chaotic are example of such circuits \cite{E}, \cite{Sprott}. 
In classical mechanics, the function $J$ is called \emph{jerk}, and corresponds to the rate of change of acceleration, or equivalently to the third-time derivative of the position $x$. 
A jerk flow so is an explicit third order differential equation as above describing the evolution of the position $x(t)$ with the time $t$ 

The following non-linear third-order differential equation is the jerk flow studied by Vaidyanathan \cite{V}: 
\begin{equation}\label{jerkequation}
\dddot{x} = -a \ddot{x} + x \dot{x}^2 -x^3 -b x + c \dot{x},
\end{equation}
where $a$, $b$ and $c$ are parameters. 
This equation generalizes the one studied by Sprott \cite{Sprott}, where $b=c=0$. 
In system form, the differential equation \eqref{jerkequation} corresponds to the 3-D jerk system
\begin{equation}\label{jerksystem}
\begin{aligned}
\dot{x} & = y\\
\dot{y} & = z\\
\dot{z} & = -a z - b x+ c y+xy^2 -x^3.
\end{aligned}
\end{equation}
In \cite{V}, Vaidyanathan shows that system \eqref{jerksystem} is chaotic when  $a=3.6$, $b=1.3$ and $c=0.1$. 
The aim of the present paper is to study this system depending on the parameters $(a,b,c) \in \R^3$ from another point of view. 

A \textit{zero-Hopf equilibrium} of a 3-dimensional autonomous differential system is an isolated equilibrium point of the system, which  has a zero eigenvalue and a pair of purely imaginary eigenvalues. In general, the \textit{zero-Hopf bifurcation} is a parametric unfolding of a $3$-dimensional autonomous differential system with a zero-Hopf equilibrium. The unfolding can exhibit different topological type of dynamics in the small neighborhood of this isolated equilibrium as the parameter varies in a small neighborhood of the origin. For more information on the zero-Hopf bifurcation, we address the reader to Guckenheimer, Han, Holmes, Kuznetsov, Marsden and Scheurle in \cite{G, GH, H, K,SM}, respectively. 

As far as we know nobody has studied the existence or non-existence of zero-Hopf equilibria and zero-Hopf bifurcations in the 3-D jerk system \eqref{jerksystem}. 
The objective of this paper is to persuade this study. 

Usually the main tool for studying a zero-Hopf bifurcation is to pass the system to the normal form of a zero-Hopf bifurcation. 
Our analysis, however, will use the \emph{averaging theory} (see Section \ref{section2} for the results on this theory needed for our study). 
The averaging theory has already been used to study Hopf and zero-Hopf bifurcations in some others differential systems, see for instance \cite{BLM, CLQ,  EL, GL, L, ML, LOV}. 

Our main results are the following. 
We first characterize the zero-Hopf equilibrium point of system \eqref{jerksystem} in Proposition \ref{prop1}. 
\begin{proposition}\label{prop1}
The differential system \eqref{jerksystem} has a  zero-Hopf equilibrium point if and only if $a=b=0$ and $c<0$. 
In this case, the zero-Hopf equilibrium is the only singular point of the system, and it is localized at the origin. 
\end{proposition}

Then we study when the  $3-D$ jerk system \eqref{jerksystem} having a zero-Hopf equilibrium point at the origin of coordinates has a zero-Hopf bifurcation producing some periodic orbit in Theorem \ref{teo1}. 

\begin{theorem}\label{teo1} 
Let $a_2,b_2,c_1,c_2,\delta \in \R$ such that $3 - \delta^2 \neq 0$ and $2 a_2 \delta^2 \neq b_2$ and set $(a, b, c) = (\e^2 a_2, \e^2 b_2, -\delta^2 + \e c_1+\e^2 c_2)$. 
Then the 3-D jerk system \eqref{jerksystem} has a zero-Hopf bifurcation at the equilibrium point localized at the origin of coordinates in the following situations: 
\begin{enumerate}

\item $\dfrac{a_2 \delta ^2+ 2 b_2}{3-\delta^2}<0$ and $\dfrac{a_2\delta^2 -b_2}{3-\delta^2}>0$. In this case three periodic orbits born at the equilibrium point when $\e\to 0$.

\item $\dfrac{a_2 \delta ^2+ 2 b_2}{3-\delta^2}<0$ and $\dfrac{a_2\delta^2 -b_2}{3-\delta^2}<0$. In this case two periodic solutions born at the equilibrium point when $\e \to 0$.

\item $\dfrac{a_2 \delta ^2+ 2 b_2}{3-\delta^2}>0$ and $\dfrac{a_2\delta^2 -b_2}{3-\delta^2}>0$. In this case one periodic orbit borns at the equilibrium point when $\e \to 0$
\end{enumerate} 
\end{theorem} 


We illustrate a case of Theorem \ref{teo1} in Fig. \ref{fig1}. 
\begin{figure}
\begin{center}
\includegraphics[scale=.3]{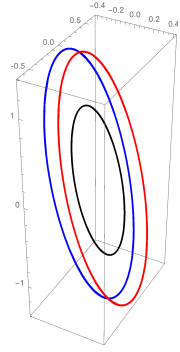}
\end{center}
\caption{Three periodic solutions emanating from the origin of coordinates. Here $\delta = 2$, $a_2 = 1$, $b_2 = 5$, $c_1 = c_2 = 0$ and $\e = 1/10$.}\label{fig1}
\end{figure}
We prove our results in Section \ref{s3}.

\begin{remark}
We prove Theorem \ref{teo1} by applying averaging theory of order $2$. 
Since we prove that the averaged function cannot be identically zero, it follows that with averaging of higher order we will not find more periodic solutions. 	
\end{remark}



\section{The averaging theory of first and second order}\label{section2}

In this section we summarize the main results on the theory of
averaging which will be used in the proof of Theorem \ref{teo1}. 
For a proof of the following theorem and more information on averaging theory, we address the reader to \cite{BL1, CLN}.

\begin{theorem}\label{averaging} 
Let $D$ be an open subset of $\R^n$, $\e_f > 0$ and consider the differential system
\begin{equation}\label{za2}
\dot x(t)=\e F_{1}(x,t)+\e^2 F_{2}(x,t)+\e^3 R(x,t,\e), 
\end{equation}
where $F_{1}, F_2:D\times\mathbb{R} \to \mathbb{R}^n$,
$R:D\times\mathbb{R}\times (-\e_f,\e _f)\to \mathbb{R}^n$ are
continuous functions, $T$--periodic in the second variable, $F_{1}(\cdot,t)\in C^1(D)$ for all $t\in \R$, $F_{1}$,
$F_{2}$, $R$ and $D_{x} F_{1}$ are locally Lipschitz with respect to $x$, and $R$ is  differentiable with respect to $\e$. 
Define $f, g:D\rightarrow \mathbb{R}^n$ as
\begin{equation}\label{ave1}
f(z)=\dfrac{1}{T} \displaystyle\int _0^{T}F_{1}(z, s)ds, 
\end{equation}
\begin{equation}\label{ave2}
g(z)=  \dfrac{1}{T} \displaystyle\int _0^{T}\left[D_z F_{1}(z, s) \cdot \displaystyle\int _0^{s} F_{1}(z, t)dt + F_2(z, s)\right]ds, 
\end{equation}
and assume that 
for an open and bounded subset $V\subset D$ and for each $\e \in (-\e_f,\e_f)\setminus \{0\}$, there exists $a_\e \in V$ such that  $f(a_{\e})+\e g(a_{\e})=0$ and $ d_B(f+\e g,V,a_\e)\neq 0$. 
Then for $|\e|>0 $ sufficiently small, there exists a $T$--periodic solution $\varphi (\cdot, \e)$ of system \eqref{za2} such that $\varphi (0, \e) = a_\e$. 
\end{theorem}

The expression $d_B(f+\e g,V,a_\e)\neq 0$ means that the Brouwer
degree of the function $f+\e g: V\to \R^n$ at the fixed point $a_{\e}$ is
not zero. We recall that a sufficient condition for this is that the Jacobian determinant of the function $f+\e g$ at $a_{\e}$ is not
zero. 
For the definition, the mentioned and other properties of Brouwer degree we address the reader to Browder's paper \cite{B}. 

If $f$ is not identically zero, then the zeros of $f+\e g$ are mainly the zeros of $f$ for $\e$ sufficiently small. In this case the previous result provides the averaging theory of first order.

If $f$ is identically zero and $g$ is not identically zero, then clearly the zeros of $f+\e g$ are the zeros of $g$. 
In this case the previous result provides the averaging theory of second order.



\section{Proofs}\label{s3}
\begin{proof}[Proof of Proposition \ref{prop1}]
$(x,0,0)$ with $x = 0$ or $x^2 = - b$, if $b\leq0$, are the singular points of system \eqref{jerksystem}. 

The characteristic polynomial of the linear part of the system at $(x,0,0)$ is 
$$
p(\lambda)= - \lambda^3  - a \lambda^2+ c \lambda - b - 3 x^2.
$$
In order to have a zero-Hopf equilibrium, we need one null and two purely imaginary, say $\pm i \delta$, with $\delta > 0$, eigenvalues. 
Imposing that 
$$
p(\lambda) = -\lambda(\lambda^2+\delta^2), 
$$ 
we obtain $a = b = 0$ and $c=-\delta^2$. 
In particular, the only zero-Hopf equilibrium is $(0,0,0)$ and we are done. 
\end{proof}
 
\begin{proof}[Proof of Theorem \ref{teo1}] 
With the parameters $(a, b, c) = (\e a_1 + \e^2 a_2, \e b_1 + \e^2 b_2, -\delta^2 + \e c_1 + \e^2 c_2)$, the $3-D$ jerk system \eqref{jerksystem} takes the form
\begin{equation}\label{jerksystem1}
\begin{aligned}
\dot{x} & =y \\ 
\dot{y} & = z \\ 
\dot{z} & = -\e \left(a_1 + \e a_2\right) z - \e \left(b_1 + \e b_2\right) x+ (-\delta^2 + \e c_1 + \e^2 c_2) y + x y^2 - x^3, 
\end{aligned}
\end{equation}
As in the proof of Proposition \ref{prop1}, when $\varepsilon = 0$, the eigenvalues at the origin of system \eqref{jerksystem1} are $0$ and $\pm i \delta$. 

With the change of variables $(x,y,z) = \left(\e X, \e Y, \e Z\right)$ and rescaling of time, system \eqref{jerksystem1} writes 
\begin{equation}\label{jerksystem2}
\begin{aligned}
\dot{X} & = Y \\ 
\dot{Y} & = Z \\ 
\dot{Z} & = -\delta^2 Y \hspace{-.1cm} + \hspace{-.1cm} \e\left(-a_1 Z -b_1 X +c_1 Y\right) \hspace{-.1cm} + \hspace{-.1cm} \e^2\left(-a_2 Z -b_2 X +c_2 Y+XY^2-X^3\right). 
\end{aligned}
\end{equation}
Now we make a linear change of variables in order to have the matrix of the linear part of system \eqref{jerksystem2} at the origin when $\e = 0$ written in its real Jordan normal form 
$$
J = 
\left( \begin{array}{ccc} 0 & -\delta &0\\ 
\delta & 0 & 0\\
0 & 0 & 0 \end{array}
\right).
$$
It is simple to see that the change 
\begin{equation}\label{changevariables}
\begin{aligned}
X & = w +\frac{v}{\delta}, \\
Y & = u, \\
Z & = - \delta v
\end{aligned}
\end{equation} 
makes what we want. 
In these new variables, system \eqref{jerksystem2} writes 
\begin{equation}\label{jerksystem3}
\begin{aligned}
\dot{u} & = - \delta v \\ 
\dot{v} & = \delta u + \e \delta \left(h_1 + \e h_2\right)\\
\dot{w} & = -\e \left(h_1 + \e h_2 \right), 
\end{aligned}
\end{equation}
with 
$$
\begin{aligned}
h_1 = h_1(u,v,w) & = \frac{b_1 v}{\delta^3} - \frac{c_1 u -b_1 w}{\delta^2} - \frac{a_1 v}{\delta}, \\
h_2 = h_2(u,v,w) & = \frac{v^3}{\delta^5} + \frac{3 v^2 w}{\delta^4} + \frac{(b_2 - u^2 + 3 w^2) v}{\delta^3} - \frac{c_2 u + (u^2 - b_2) w - w^3}{\delta^2} - \frac{a_2 v}{\delta}. 
\end{aligned}
$$
Now we pass the differential system \eqref{jerksystem3} to cylindrical coordinates $(r, \theta,w)$ defined by $u=r\cos\theta$, $v=r \sin\theta$, $w=w$, obtaining 
\begin{equation}\label{interm}
\begin{aligned}
\dot{\theta} & = \delta + \e \delta \frac{\cos \theta}{r} \left(h_1 + \e h_2\right), \\
\dot{r} & = \e \delta \sin \theta \left(h_1 + \e h_2\right), \\
\dot{w} & = - \e \left(h_1 + \e h_2\right), 
\end{aligned}
\end{equation}
with $h_1 = h_1\left(r \cos \theta, r\sin \theta, w\right)$ and $h_2 = h_2\left(r\cos \theta, r \sin \theta, w\right)$. 
By introducing $\theta$ as the new independent variable, we obtain 
\begin{equation}\label{theta}
\begin{aligned}
\frac{dr}{d\theta} & = \frac{\e  (h_1 + \e h_2)}{r+\e \cos \theta (h_1+\e h_2)} r \sin \theta \\ 
& = \e h_1 \sin \theta + 2 \sin \theta \e^2 \frac{h_2 r - h_1^2 \cos \theta}{r} + O(\e^3)\\
\frac{d w}{d \theta} & = - \frac{\e (h_1 + \e h_2)}{r + \e \cos \theta (h_1 + \e h_2)}\frac{r}{\delta} \\ 
& = -\e \frac{h_1}{\delta} - 2 \e^2 \frac{h_2 r- h_1^2 \cos \theta}{r \delta} + O(\e^3)
\end{aligned}
\end{equation}
In the notation of Theorem \ref{averaging}, by taking $t = \theta$, $T = 2\pi$ and $z = (r, w)^{tr}$, we have 
$$
F_1(r, w, \theta) = h_1 \left(
\begin{array}{c}
\sin \theta \\
-1/\delta 
\end{array} 
\right), \ \ \ \ F_2(r, w, \theta) = 2 \frac{h_2 r - h_1^2 \cos \theta}{r} \left( 
\begin{array}{c}
\sin \theta \\
-1/\delta 
\end{array} 
\right). 
$$
We calculate $f(r, w)$ obtaining 
$$
f(r,w) = \frac{1}{2 \pi} \int_0^{2 \pi} F_1(r,w, \theta) d\theta =  \left(
\begin{array}{c}
\dfrac{r}{2 \delta^3} (b_1 - a_1 \delta^2) \\
\dfrac{-b_1 w}{\delta^3} 
\end{array}
\right)
$$
The solutions of $f(r,w) = 0$ with $b_1 - a_1\delta^2 \neq 0$ are contained in $r = 0$, and so they are not allowed, because $r$ must be positive. 
On the other hand, if $b_1 - a_1 \delta^2 = 0$, the zeros of $f(r, w)$ are not isolated. 
In particular we can not apply the averaging of first order.  
We pass then to the averaging of second order, assuming $f \equiv 0$. 
This makes $a_1 = b_1 = 0$. 
We now calculate $g(r, w)$: 
$$
\begin{aligned}
g(r, w) & = \frac{1}{2 \pi} \int_0^{2 \pi} \left(D_{(r, w)} F_1(r,w,\theta) \int_0^\theta F_1(r,w,s) ds + F_2(r,w,\theta)  \right) d\theta \\
& = \frac{1}{2 \pi} \int_0^{2 \pi} \left(
\frac{c_1^2 r }{2 \delta^5}\left(
\begin{array}{c} 
 \delta \cos \theta \sin^3 \theta \\
-\cos \theta \sin^2 \theta
\end{array} 
\right)
+ F_2(r, w, \theta)
  \right) d \theta \\
& = \frac{1}{2 \delta^5} \left( 
\begin{array}{c} 
r\left((3 - \delta^2) r^2 + 4 b_2 \delta^2 - 4 a_2 \delta^4 + 12 \delta^2 w^2 \right)/4 \\
-w \left((3 - \delta^2) r^2 + 2 b_2 \delta^2 + 2 \delta^2 w^2 \right)
\end{array}
\right). 
\end{aligned}
$$
We analyze the solutions of $g(r, w) = 0$, with $r > 0$. 
We first observe that in order to obtain isolated solutions according to Theorem \ref{averaging}, we ought to have $\delta^2 \neq 3$. 
With this assumption in force, we readily obtain the following two group of solutions 
$$
r^2 = \frac{4 (a_2 \delta^2 - b_2) \delta^2}{3 - \delta^2}, \ \ \ \ w = 0, 
$$
or 
$$
r^2 = \frac{- 4 (a_2 \delta^2 + 2 b_2) \delta^2}{5 (3 - \delta^2)}, \ \ \ \ w^2 = \frac{2 a_2 \delta^2 - b_2}{5}. 
$$ 
The Jacobian determinant of $g$ at the solutions above take the values 
$$
- \frac{(a_2 \delta^2 - b_2) (2 a_2 \delta^2 - b_2)}{\delta^6} 
$$ 
and 
$$ 
-\frac{2 (a_2 \delta^2 + 2 b_2) (2 a_2 \delta^2 - b_2)}{5 \delta^6}, 
$$
respectively. 



If $\dfrac{a_2\delta^2 -b_2}{3-\delta^2}>0$, $\dfrac{a_2\delta^2 +2b_2}{3-\delta^2}<0$ and $2 a_2 \delta^2 - b_2 \neq 0$, we have the isolated solutions $(r_1, w_1) = \left(2 \delta \sqrt{\dfrac{a_2 \delta^2 - b_2}{3 - \delta^2}}  ,0\right)$, $(r_2, w_2) = \left(2 \delta \sqrt{-\dfrac{a_2 \delta^2 + 2 b_2}{5 (3 - \delta^2)}}, \sqrt{\dfrac{2 a_2 \delta^2 - b_2}{5}} \right)$ and $(r_3, w_3) = (r_2, -w_2)$. 

If $\dfrac{a_2\delta^2 -b_2}{3-\delta^2}>0$, $\dfrac{a_2\delta^2 +2b_2}{3-\delta^2}\geq 0$ and $2 a_2 \delta^2 - b_2 \neq 0$    then we have just the isolated solution $(r_1, w_1)$.

If   $\dfrac{a_2\delta^2 -b_2}{3-\delta^2}\leq 0$, $\dfrac{a_2\delta^2 +2b_2}{3-\delta^2}<0$  and 
$2 a_2 \delta^2 - b_2 \neq 0$ we will have just the solutions $(r_2, w_2)$ and $(r_3, w_3)$. 
\smallskip

Now Theorem \ref{averaging} guarantees that, for $\e$ sufficiently small, to each root $(r_i, w_i)$ of $g$ it corresponds a periodic solution with period $2 \pi$ of system \eqref{theta} of the form $(r(\theta,\e), w(\theta,\e))$, with $(r(0,\e),w(0,\e)) = (r_i,w_i)$. 
Corresponding to this one, system \eqref{interm} has the periodic solution of certain period $T_\e$ $(\theta(t,\e),r(t,\e),w(t,\e))$ satisfying $\left(\theta(0,\e), r(0, \e), w(0, \e)\right) = (0, r_i, w_i)$. 
Then system \eqref{jerksystem3} has the periodic solution of period $T_\e$ 
$$
\left(u(t,\e), v(t,\e), w(t, \e) \right) = \left( r(t,\e)\cos \theta(t, \e),  r(t,\e)\sin \theta(t, \e), w(t,\e) \right),
$$
for $\e$ sufficiently small, with $(u(0,\e), v(0,\e), w(0,\e)) = (r_i, 0, w_i)$. 
Now we apply the change of variables \eqref{changevariables} to this one and obtain for small $\e$ the periodic solution $\left(X(t, \e), Y(t,\e), Z(t,\e)\right)$ of system \eqref{jerksystem2} with the same period, such that  $(X(0,\e,Y(0,\e),Z(0,\e)) = (w_i + r_i/\delta, r_i, 0)$. 
Finally, for $\e\neq 0$ sufficiently small, system \eqref{jerksystem1} has the periodic solution $\left(x(t, \e), y(t, \e), z(t, \e)\right) = \left(\e X(\theta),\e Y(\theta), \e Z(\theta)\right)$, with $\left(x(0, \e), y(0, \e), z(0, \e)\right) = \e  (w_i + r_i/\delta, r_i, 0)$, and that clearly tends to the origin of coordinates when $\e\rightarrow 0$. 
Thus, this is a periodic solution emanating from the zero-Hopf bifurcation point located at the origin of coordinates when $\e = 0$. 
This concludes the proof of Theorem \ref{teo1}. 
\end{proof}


\section*{Acknowledgments}
We thank Prof. Luis Fernando Mello for calling our attention to the system studied here. 
The first author was partially supported by the grant 2017/00136-0, S\~ao Paulo Research Foundation (FAPESP).



\begin{thebibliography}{99}
	


\bibitem{BL1} {\sc A. Buic\u{a} and J. Llibre},
{\it Averaging methods for finding periodic orbits via Brouwer
	degree}, {Bull. Sci. Math.} {\bf 128} (2004), 7--22. 


\bibitem{B}{\sc F.E. Browder}, {\it Fixed point theory and nonlinear problems}, 
Bull. Amer. Math. Soc. (1983), 1--39. 

\bibitem{BLM}{\sc C. Buzzi, J. Llibre and J. Medrado}. {\it  Hopf and zero-Hopf bifurcations in the Hindmarsh-Rose system}. {Nonlinear Dyn.} {\bf 83(3)} (2016), 1549--1556. 

\bibitem{CLN}{\sc M.R. C\^andido, J. Llibre and D.D. Novaes}, 
\emph{Persistence of periodic solutions for higher order perturbed differential systems via Lyapunov-Schmidt reduction}, 
Nonlinearity \textbf{30} (2017), 3560--3586. 

\bibitem{CLQ}{\sc V. Castellanos, J. Llibre, I. Quilantán} {\it Simultaneous periodic orbits bifurcating from two zero-hopf equilibria in a tritrophic food chain model.}{ J. Appl. Math. Phys.} {\bf1(7)},  (2013) 31--38.

\bibitem{C} {\sc L.O. Chua},
{\it The Genesis of Chua's Circuit}, {Archiv Elektronik \"{U}bertragungstechnik} {\bf 46} (1992), 250--257. 

\bibitem{E} {\sc R. Eichhorn, S.J. Linz and P. H\"{a}nggi}, 
{\it Transformations of nonlinear dynamical systems to  jerk motion and its application to minimal chaotic flows}, {Phys. Rev. E } {\bf 58} (1998), 7151--7164. 


\bibitem{EL} {\sc R. Euz\'ebio and J. Llibre}, 
{\it Zero-Hopf bifurcation in a Chua system}, {Nonlinear Anal. Real World Appl. } {\bf 37} (2017), 31--40. 



\bibitem{GL} {\sc J. M. Ginoux, J. Llibre}
{\it Zero-Hopf bifurcation in the volterra-gause system of predator-prey type},
{Math. Methods Appl. Sci.}  {\bf 40 (18) } (2017), pp. 7858-7866

\bibitem{G} {\sc J. Guckenheimer}, {\it  On a Codimension Two Bifurcation}, in: Lectures Notes in Math., {\bf  898} (1980),  99--142.

\bibitem{GH} {\sc J. Guckenheimer, P. Holmes}, {\it Nonlinear Oscillations, Dynamical Systems, and Bifurcations of Vector Fields}. Revised and
Corrected Reprint of the 1983 Original, in: Applied Mathematical Sciences, vol. 42, Springer-Verlag, New York, 1990.

\bibitem{H} {\sc M. Han},  {\it Existence of periodic orbits and invariant tori in codimension two bifurcation of three dimensional systems}, { J.
Sys. Sci. Math. Sci}.  {\bf 18} (1998) 403--409.

\bibitem{K} {\sc  Y.A. Kuznetsov}, {\it Elements of Applied Bifurcation Theory}, third ed., Spring-Verlag, 2004.

\bibitem{L} {\sc J. Llibre}, {\it  Periodic orbits in the zero-hopf bifurcation of the rossler system}. {Roman. Astron. J}.{\bf 24(1)} (2014) 49--60.

\bibitem{ML} {\sc J. Llibre, A. Makhlouf,} {\it Zero-Hopf periodic orbit of a quadratic system of differential equations obtained from a third-order differential equation. }{Differ. Equ. Dyn. Syst.} {\bf 27} (2019) 75--82.
	
\bibitem{LOV} {\sc  J. Llibre, R.D.S. Oliveira, C. Valls,} {\it On the integrability and the zero-Hopf bifurcation of a Chen-Wang differential system}. {Nonlinear Dyn.} {\bf 80(1–2)}  (2015) 353--361.


	
\bibitem{SM} {\sc J. Scheurle, J. Marsden}, {\it Bifurcation to quasi-periodic tori in the interaction of steady state and Hopf bifurcations}, SIAM
{J. Math. Anal.} {\bf 15} (1984) 1055--1074.

\bibitem{Sprott} {\sc J.C. Sprott}, 
{\it Some simple chaotic jerk functions}, {Amer. J. Phys.} {\bf 65} (1997), 537--543. 

\bibitem{V} {\sc S. Vaidyanathan}, 
{\it A new 3-D jerk chaotic system with two cubic nonlinearities and its adaptive backstepping control}, {Archives of Control Sciences} {\bf 27} (2017), 409--439. 


\end{thebibliography}
\end{document}